\documentclass[11pt]{amsart}

\usepackage{epigamath}

\usepackage[english]{babel}

\usepackage{amscd, amssymb, amsmath, mathrsfs, amsthm}
\usepackage[utf8]{inputenc}
\usepackage{tikz-cd}
\usetikzlibrary{decorations.pathreplacing}
\usepackage{booktabs}
\usepackage{color}

\usepackage{multicol}

\usepackage{microtype}

\usepackage{xypic}
\usepackage{graphicx}
\usepackage{fancyhdr} 
\usepackage{enumerate}

\newtheorem{thm}[subsection]{Theorem}
\newtheorem*{thmstar}{Theorem}
\newtheorem{lemma}[subsection]{Lemma}
\newtheorem{cor}[subsection]{Corollary}
\newtheorem{prop}[subsection]{Proposition}
\theoremstyle{definition}
\newtheorem{defn}[subsection]{Definition}
\newtheorem{ex}[subsection]{Example}
\theoremstyle{remark}
\newtheorem{rem}[subsection]{Remark}

\numberwithin{equation}{subsection}

\DeclareMathOperator{\Spec}{Spec}
\DeclareMathOperator{\Proj}{Proj}
\DeclareMathOperator{\Ann}{Ann}
\DeclareMathOperator{\Fitt}{Fitt}
\DeclareMathOperator{\Hom}{Hom}

\DeclareMathOperator{\Def}{Def}

\def\mapleft#1{\mathrel{%
\smash{\mathop{\longleftarrow}\limits^{#1}}}}
\def\mapright#1{\mathrel{%
\smash{\mathop{\longrightarrow}\limits^{#1}}}}
\def\maplongright#1{\mathrel{%
\mathop{\relbar\joinrel\relbar\joinrel\relbar\joinrel\longrightarrow}\limits^{#1}}}
\newcommand{\PP}{\mathbb{P}}
\newcommand{\ZZ}{\mathbb{Z}}
\renewcommand{\AA}{\mathbb{A}}
\newcommand{\eps}{\varepsilon}
\newcommand{\dd}[1]{\frac{\partial}{\partial #1}}

\newcommand{\m}{\mathfrak{m}}
\newcommand{\n}{\mathfrak{n}}
\newcommand{\CM}{\mathrm{CM}}
\newcommand{\TC}{\mathrm{H}}
\newcommand{\ra}{\longrightarrow}
\newcommand{\HC}{\mathrm{H}}
\newcommand{\Hilb}{\operatorname{Hilb}^{3t+1}}

\newcommand{\calU}{\mathscr{U}}
\newcommand{\calK}{\mathscr{K}}

\newcommand{\calO}{\mathscr{O}}
\newcommand{\calF}{\mathscr{F}}

\newcommand{\calM}{\mathscr{M}}
\newcommand{\calN}{\mathscr{N}}

\EpigaVolumeYear{5}{2021}
\EpigaArticleNr{10}
\ReceivedOn{June 13, 2019}
\InFinalFormOn{February 18, 2021}
\AcceptedOn{March 5, 2021}

\title{The space of twisted cubics}
\titlemark{The space of twisted cubics}

\author{Katharina Heinrich}
\address{Department of Mathematics, KTH Royal Institute of Technology, SE 100 44 Stockholm, Sweden.}
\email{kchal@kth.se}

\author{Roy Skjelnes}
\address{Department of Mathematics, KTH Royal Institute of Technology, SE 100 44 Stockholm, Sweden.}
\email{skjelnes@kth.se}

\author{Jan Stevens}
\address{Department of Mathematical Sciences, Chalmers University of
Technology and University of Gothenburg, SE 412 96 G\"oteborg, Sweden.}
\email{stevens@chalmers.se}

\authormark{K. Heinrich, R. Skjelnes, and J. Stevens}

\AbstractInEnglish{We consider the Cohen-Macaulay compactification of the space of twisted cubics in $\PP^n$. This compactification is the fine moduli scheme representing the functor of CM-curves with Hilbert polynomial $3t+1$. We show that the moduli scheme of CM-curves in $\PP^3$ is isomorphic to the twisted cubic component of the Hilbert scheme. We also describe the compactification for twisted cubics in $\PP^n$.}



\TitleInFrench{L'espace des cubiques gauches}

\AbstractInFrench{Nous considérons la compactification de Cohen-Macaulay de l'espace des cubiques gauches de $\PP^n$. Cette compactification est l'espace de modules fin qui représente le foncteur des courbes CM avec polynôme de Hilbert $3t+1$. Nous montrons que le schéma des modules des courbes CM dans $\PP^3$ est isomorphe à la composante des cubiques gauches du schéma de Hilbert. Nous décrivons également la compactification pour les cubiques gauches de $\PP^n$.}

\begin{document}



\maketitle

\begin{prelims}

\DisplayAbstractInEnglish

\bigskip

\DisplayKeyWords

\medskip

\DisplayMSCclass

\bigskip

\languagesection{Fran\c{c}ais}

\bigskip

\DisplayTitleInFrench

\medskip

\DisplayAbstractInFrench

\end{prelims}


\newpage

\setcounter{tocdepth}{1} 

\tableofcontents


\section{Introduction}

The main purpose of the present article is to
provide a natural and functorial compactification of the space of  twisted cubics in $\PP^3$.  The representing object is a smooth and irreducible projective scheme that is isomorphic to the twisted cubic component of the Hilbert scheme.

A twisted cubic is a smooth rational curve of degree three.  The family of such curves in $\PP^3$ is twelve dimensional. Finding a compactification of the parameter space 
with  control over the degenerations is a  fundamental question.
An immediate answer is provided by the Hilbert scheme. In their celebrated
work  \cite{PieneSchlessinger} Piene and Schlessinger described the Hilbert scheme 
$\Hilb_{\PP^3}$ of space curves with Hilbert polynomial $3t+1$. It consists of two smooth
components. One component $\HC$, the twisted cubic component,  contains the 
twisted cubics as an open subset. The other component has as  general member 
the union of a planar cubic and a point outside the plane.

An appealing feature of the Hilbert scheme  is that it
represents a functor, and thereby comes with universal defining properties. It gives a description of the families of curves including their degenerations.
Its disadvantage is clear from the description  
above --- it parametrizes also objects that are geometrically quite different from twisted cubics. 
Kontsevich's moduli space of stable maps $\overline{\calM_{0,0}}(\PP^3,3)$ is  another functorial compactification of the space of twisted cubics \cite{Kontsevich, ChungKiem}. The representing stack has been a marvelous tool for enumerative problems, but the geometry  of this compactification is quite different from the twisted cubic component $\HC$.

An interpolation between the Hilbert scheme and the Kontsevich moduli space was introduced by H\o nsen \cite{Honsen}.  We refer to it as the space of CM-curves.  A Cohen-Macaulay scheme $X$ of pure dimension one,  together with a finite morphism $X\ra \PP^n$ that generically is an immersion is  a CM-curve (see Section \ref{section_CM}). The functor  of CM-curves in  $\PP^n$ with a fixed Hilbert polynomial is represented by a proper algebraic space (\cite{Honsen}, \cite{Heinrichrep}).

In the present article we use the space of CM-curves 
with Hilbert polynomial $3t+1$
to compactify the space of twisted cubics. The most interesting case is  that of curves in $\PP^3$. 
Our main result  (see Theorems \ref{smooth} and \ref{main}) gives in particular a modular description of the twisted cubic component $\HC$:
\begin{thmstar}
Let $\CM^{3t+1}_{\PP^3}$ be the space of CM-curves in $\PP^3$ (over an algebraically closed field of arbitrary characteristic) with Hilbert polynomial $3t+1$, and let $\HC \subset \Hilb_{\PP^3}$ denote the twisted cubic component of the Hilbert scheme.
The space $\CM^{3t+1}_{\PP^3}$ is a smooth  and irreducible 
projective scheme
of dimension 12. We have an  isomorphism 
\[
\phi \colon \CM^{3t+1}_{\PP^3} \ra \HC.
\]
\end{thmstar}

The map $\phi$ yielding the isomorphism is obtained by giving 
the image of the finite map $i\colon X  \ra \PP^3$ the scheme structure
determined by the $0^\mathrm{th}$ Fitting ideal $\Fitt^0(i_*{\calO}_X)$. This construction 
is discussed in detail in \cite{Teissier}. It commutes with base change, but the
resulting structure depends on the target of the morphism.
Another important ingredient in our story is an explicit
description of CM-curves in the plane. The space of CM-curves in $\PP^2$ having Hilbert polynomial $3t+1$ is identified with the space parametrising
families of plane cubic curves, together with a section that passes through singular points of the family. We use  a duality result by de~Jong and van Straten relating Cohen-Macaulay ring extensions with the endomorphism ring of dual modules \cite{DucoTheo}, which allows us to reconstruct the curve $X$ from its image and the
singular section.
For projective $n$-spaces with $n>3$ we 
identify a blow up of $\CM^{3t+1}_{\PP^n}$ with the twisted cubic component of $\Hilb_{\PP^n}$, based on the
description of the Hilbert space in \cite{ChungKiem}.

Other compactifications of the space of twisted cubics are known, see \emph{e.g.} \cite{EPS, vainsencherxavier, ChungKiem}. 
The connections between different compactifications
have been studied by \cite{ChungKiem} and by
\cite{Chen}, the latter starting from the Kontsevich
space.
Freiermuth and Trautmann \cite{FreiermuthTrautmann} studied (in characteristic 0) the moduli space of  stable sheaves on  $\PP^3$ supported on cubic curves
and found that it has two smooth components, one of which is isomorphic to the twisted cubic component $\HC$ of the Hilbert scheme, see Appendix \ref{app:FT}.
Their identification also uses Fitting ideals.
Our results are  valid in arbitrary characteristic; to this end we  extend the results of \cite{PieneSchlessinger}
to  also include  characteristic 2 and 3, see Appendix \ref{app:PS}.

In general one cannot expect that a component of the space of CM-curves is isomorphic to a 
component of the corresponding Hilbert scheme --- this already fails for $\CM^{3t+1}_{\PP^n}$ with $n>3$.
Therefore the situation with Hilbert polynomial $3t+1$ in $\PP^3$ appears to be quite special.
Only the case with Hilbert polynomial $2t+2$ is 
similar and can be studied analogously. 
In characteristic zero it is known (\cite{Harris, CCN}) that the Hilbert scheme $\operatorname{Hilb}^{2t+2}_{\PP^3}$ consists of two smooth components. With the arguments in this article one can show that $\CM^{2t+2}_{\PP^{3}}$ is isomorphic to the component whose generic point corresponds to a pair of skew lines. 
In both cases only the simplest type of non-isomorphism
occurs. This is no longer true for the case of quartic elliptic
space curves. The irreducible component $\HC$ of the Hilbert scheme 
$\operatorname{Hilb}^{4t}_{\PP^3}$  containing the points 
corresponding to these curves is explicitly known  by the work of
Avritzer and Vainsencher  \cite{AV}; the Hilbert scheme consists of 
two components \cite{Gotzmann2008, LellaRoggero}.
We expect that there exists a map
from the component $\HC$ to the space of
CM-curves.  
The picture is much more complicated for rational
normal curves of degree 4, with Hilbert polynomial $4t+1$.
We hope that the study of the space of CM-curves can 
help in the study of the Hilbert scheme.

\subsection*{Acknowledgments}

We are grateful for the comments of the anonymous referee that helped us to clarify certain arguments, and  to improve the presentation of the article.

\section{Cohen-Macaulay curves}
\label{section_CM}

In this section we 
define the space of CM-curves in general, 
and thereafter determine the dimension and prove smoothness  when the Hilbert polynomial is 
$3t+1$.

\subsection{The image of a finite map} The schematic image of a quasi-compact morphism of schemes $f\colon X\ra Y$ is the closed subscheme of $Y$ determined by the kernel of the induced map $\calO_Y \ra f_*{\calO}_X$. When $f\colon X\ra Y$ is a finite morphism the ideal defining the schematic image is also given by the annihilator ideal $\Ann_Y(f_*{\calO}_X)$.

\subsection{Fitting ideals}
The $n^\mathrm{th}$ Fitting ideal sheaf of a
coherent $\calO_X$-module $E$ is denoted by $\Fitt^n_X(E)$, or 
simply $\Fitt^n(E)$. We refer to \cite{Northcott} or \cite{Eisenbud} for definitions 
and their basic properties. We will be particularly interested in the $0^\mathrm{th}$ Fitting ideal, 
inspired by the discussion by Teissier in \cite{Teissier}. The ideal sheaf $\Fitt^0_X(E)$ is obtained from the maximal minors of the matrix in a presentation of the module $E$.

\subsection{CM-curves}\label{CM curve} We recall the notion of 
CM-curves as introduced by H\o nsen \cite{Honsen}. Let $\PP^n$ denote the
projective $n$-space over a fixed algebraically closed field $k$. A finite morphism 
$i\colon X\ra \PP^n$ is a
{\em CM-curve} if 
\begin{enumerate}
\item the scheme $X$ is Cohen-Macaulay and is of pure dimension one,
\item the morphism $i$ is, apart from a finite set, an isomorphism onto its 
schematic image.
\end{enumerate}
The Hilbert polynomial of a CM-curve $i\colon X\ra\PP^n$ is the Hilbert 
polynomial of the coherent sheaf $i_*{\calO}_X$ on $\PP^n$.

\subsection{Functor of Cohen-Macaulay curves}
Let $p(t)=at+b$ be a numerical polynomial.
For a scheme $S$ we let $\CM^{p(t)}_{\PP^n}(S)$ denote 
the set of isomorphism classes
of pairs $(X,i)$ where $i\colon X\ra \PP^n_S=\PP^n\times S$ is a finite morphism of 
$S$-schemes, and where $X\ra S$ is flat, and for every  geometric point in $S$ the fiber  
is a CM-curve with Hilbert polynomial $p(t)=at+b$.
Two pairs
$(X,i)$ and $(X',i')$ 
are isomorphic if there exists an $S$-isomorphism $\alpha \colon X\ra X'$ commuting with the 
maps $i$ and $i'$. 
We have that $\CM^{p(t)}_{\PP^n}$ 
is a contravariant functor from the category of locally Noetherian schemes, over our fixed field $k$, 
to Sets.

\begin{thm}[H\o nsen, Heinrich]\label{theorem-HH}
The functor $\CM^{p(t)}_{\PP^n}$ of CM-curves in $\PP^n$ having a fixed Hilbert 
polynomial $p(t)=at+b$, is representable by a proper algebraic space.
\end{thm}

Theorem \ref{theorem-HH} is proven  more generally in \cite{Honsen} for schemes defined over 
a fixed field  of arbitrary characteristic,
not necessarily algebraically closed. A different construction and proof is given in \cite{Heinrichrep} where  it is shown that
the functor is representable by a proper algebraic space of finite type over $\Spec \ZZ$.

\subsection{Tangent space to $\CM$}\label{tangentspace}
With $\PP^n$ over a field $k$ and the polynomial $p(t)$ fixed, 
we write $\CM(R)$ for $\CM^{p(t)}_{\PP^n}(\Spec(R))$ for any $k$-algebra $R$.
Given a pair $\xi=(X,i)$ in $\CM(k)$ we 
define the  functor $\CM_\xi$ \cite[\S3.4]{Honsen} on the category of local artinian $k$-algebras by
\[
\CM_{\xi}(R)=\{(X_R,i_R)\in \CM(R) \mid (X_R,i_R)\text{ maps to }\xi\in \CM(k) \}\;.
\]
This is in fact a deformation functor. 

For a morphism $i\colon X \ra Y$ Buchweitz
\cite{buchweitz} considers six deformation functors, one of them being $\Def_{X/Y}$,
the functor of deformations of $X$ \textit{over}
$Y$,
where one deforms $X$ and the morphism $i$,
but not the target $Y$. 
Its
tangent space is $T^1_{X/Y}=\mathbb{E}\text{xt}^1_{\calO_X}(\mathbb{L}^{\bullet}_{X/Y},\calO_X)$, where 
$\mathbb{L}^{\bullet}_{X/Y}$ is the 
cotangent complex \cite{Illusie}. 

The functor
$\CM_{\xi}$ is the functor $\Def_{X/\PP^n}$ of deformations of $X$ over $\PP^n$,
the completion of the local ring of the point $(X,i) \in \mathrm{CM}(k)$ is the versal deformation
of $X$ over $\PP^n$ and the tangent space is 
$\CM_\xi(k[\eps])$, where $\eps^2=0$. In our specific situation we shall compute this tangent space by directly describing infinitesimal deformations of $(X,i)$.

\subsection{Conductor}\label{conductor}Let $D=i(X)$ denote the schematic image of a  
CM-curve $i\colon X\ra \PP^n$. We have the short exact sequence  
\begin{equation}\label{exact-sequence}
0\ra \calO_D \ra i_*{\calO}_X \ra \calK=i_*{\calO}_X/\calO_D \ra 0    
\end{equation}
of sheaves on $D$. The {\it target double point 
scheme} \cite[p.204]{KLU} of $i\colon X\ra \PP^n$ is the closed subscheme $N_X\subseteq D$ 
defined by the ideal sheaf
\[
\label{N_X}\Ann_D(\calK)=\Ann_D(i_*{\calO}_X/\calO_D).
\]
This ideal sheaf equals the conductor of $\calO_D$ in   $i_*\calO_X$. 
By definition the underlying set $|N_X|$ of points in $D$ is a finite set, and the open set $D\setminus N_X$ is isomorphic to $X\setminus i^{-1}(N_X)$.

\subsection{Twisted cubics}  The easiest example of a space of CM-curves is that
of plane curves of degree $d$, which is a space that coincides with the Hilbert scheme. One of the easiest interesting, non-planar examples is the case of CM-curves with Hilbert polynomial $3t+1$. In what follows we will assume that $\PP^n$ is the projective $n$-space defined over an algebraically closed field.
 
\begin{prop}\label{not immersion} Let $i\colon X\ra \PP^n$ 
be a CM-curve where 
$i$ is not a closed immersion. Assume that the Hilbert polynomial of $i_*{\calO}_X$ 
is $3t+1$. Then the image $D=i(X)$ is a plane cubic curve, 
and there is exactly one point $P\in D$ over which the morphism $i$ fails to be an isomorphism. 
\end{prop}

\begin{proof} The Hilbert polynomial of the skyscraper sheaf $\calK=i_*{\calO}_X/
\calO_D$ is a positive constant $\ell> 0$; it is non-zero since by assumption the map $i$ is not a closed immersion. The Hilbert polynomial of $D$ is then $p(t)=3t+1-\ell. $
Hartshorne proved \cite[Theorem 3.1]{space_curves} that the  arithmetic genus $p_a=1-p(0)$ of 
the space curve $D\subseteq \PP^3$ is bounded above by $p_a\leq \frac12 (d-1)(d-2)$, where $d$ 
is the degree of $D$, with equality if and only if $D$ is a plane curve. The same proof applies for $D\subset \PP^n$, where $n\geq 3$ is arbitrary.
It follows that $\ell =1$.

The 
support of $\calK$ is then one point on $D$. The equality $p_a=\frac12 (d-1)(d-2)$ implies that the curve $D$ is planar.
\end{proof}

\begin{lemma}\label{embedded}
 Let $i\colon X \ra \PP^n$ be a CM-curve with Hilbert polynomial $3t+1$. Then $H^1(X,i^*\calO_{\PP^n}(1))=0$
 and $\dim_k H^0(X,i^*\calO_{\PP^n}(1))=4$. 
\end{lemma}

\begin{proof} 
Assume first that $i$ is a closed immersion.
Define the index of speciality $e(X)$ of a  Cohen-Macaulay curve $X\subset\PP^n$ 
as the largest integer $e$ such that $H^0(X,\omega_X(-e))\neq0$,  where $\omega_X$ is the dualizing sheaf. E.~Schlesinger  proves that 
$-2\leq e(X)\leq \deg(X)-3$ \cite[Corollary~3.4]{Schl}. This corollary applies as a (locally)
Cohen-Macaulay curve is a quasi ACM subscheme of dimension 1  \cite[Corollary~2.7]{Schl}.
In our case $e(X)\leq 0$ and therefore 
$0=H^0(X,\omega_X(-1))=H^1(X,\calO_X(1))$, where $\calO_X(1)=i^*\calO_{\PP^n}(1)$. As the Hilbert polynomial is $3t+1$ we then also have that
$\dim_k H^0(X,\calO_X(1))=4$.

If the morphism $i$ is not a closed immersion then by Proposition \ref{not immersion} the image of $X$ is a planar cubic $D$. Tensoring the exact sequence \eqref{exact-sequence} with $\calO_{\PP^n}(1)$ gives the exact sequence
\[ 
0\ra \calO_D(1) \ra i_*{\calO}_X(1) \ra \calK \ra 0
\]
of sheaves on $\PP^n$. We have that $\dim_k H^0(D,\calO_D(1))=3$ and that $H^1(D,\calO_D(1))=0$ as $D$ is a planar cubic. It then follows from the exact sequence that $H^1(\PP^n, i_*{\calO}_X(1))=0$
and that $\dim_k H^0(\PP^n,i_*{\calO}_X(1))=4$.
The projection formula gives that $i_*\calO_X(1)=i_*(i^*\calO_{\PP^n}(1))$, and since the morphism $i$ is finite the result follows. 
\end{proof}

\begin{prop}
\label{regular}Let $i_R\colon X_R \ra \PP^n_R$ be a CM-curve with Hilbert polynomial $3t+1$.  Assume that $R$ is a local ring.
\begin{enumerate}
    \item We have that $H^0(X_R, i_R^*\calO_{\PP^n_R}(1))$ is a free $R$-module of rank 4.
    \item Let us fix a basis of $H^0(X_R, i_R^*\calO_{\PP^n_R}(1))$ and let $j_R\colon X_R \ra \PP^3_R=\Proj(R[y_0, y_1, y_2, y_3])$ be the induced morphism. Then $j_R$  is a closed immersion.
    \item There is a finite morphism  $h_R \colon j_R(X_R) \ra \PP^n_R$ induced by a linear map $R[x_0, \ldots, x_n] \ra R[y_0, \ldots, y_3]$, such that $i_R=h_R\circ j_R$. 
    \item The pairs $(X_R,i_R)$ and $(j_R(X_R),h_R)$ are equal as points of the CM-functor. 
\end{enumerate}
\end{prop}

\begin{proof}
Let $i\colon X \ra \PP^3$ be the curve over the closed point of $\Spec(R)$. By Lemma \ref{embedded} we have 
that $H^r(X,i^*\calO_{\PP^n}(1))=0$ for $r>0$. Let $f_R \colon X_R \ra \Spec (R)$ be the structure map. It follows 
by cohomology and base change \cite[Corollaire~7.9.9]{EGA} that
${f_R}_*i_R^*\calO_{\PP^n_R}(1)$ is free, and therefore by Lemma \ref{embedded} free of rank 4.  

A choice of basis $(s_0,s_1,s_2,s_3)$ of global sections of $j_R^*\calO(1)$ gives a morphism
$$j_R\colon X_R \ra \PP^3_R=\Proj(R[y_0, y_1, y_2, y_3])$$
such that  $s_i=j_R^*(y_i)$, $i=0, \ldots, 3$. 
The coordinates
$(x_0:\dots:x_n)$ on $\PP^n_R$ give rise to global sections $i_R^*(x_i)$ of $i_R^*\calO_{\PP^n_R}(1)$, which can be expressed
linearly in the basis $(s_0,s_1,s_2,s_3)$. This defines the map $h_R$. 

To show that the finite morphism $j_R$ is a closed immersion it suffices to show that the induced map $j\colon X \ra \PP^3$ over the closed point of $\Spec(R)$ is a closed immersion. If $j$ is not a closed immersion then we have by Proposition \ref{not immersion} that the image is contained in a plane in $\PP^3$. This contradicts the fact that the global sections $s_0=j^*(y_0)$, \dots, $s_3=j^*(y_3)$ are linearly independent.
It follows from the construction that the pairs $(j_R(X_R),h_R)$ and $(X_R,i_R)$ are isomorphic.
\end{proof}

\subsection{Tangent space calculations}The purpose of the rest of the present section is to prove that the CM-spaces with Hilbert polynomial $3t+1$ are smooth. Smoothness is proven by showing that the dimension of the tangent space at the most special point equals the dimension of a general, smooth point. 
Later we will see that the most special situation is a projection of the curve
$X \subset \PP^3$ 
given by the  graded ideal $I=(u^2, uy-x^2,xu) \subset k[x,y,u,w]$.
Thus, for the moment we focus on this particular ideal. The canonical  morphism $k[x,y,w] \ra k[x,y,u,w]/I$ of graded rings determines the finite morphism 
\begin{equation}\label{map_to_P2}
i_1 \colon X\ra \PP^2 =\Proj(k[x,y,w]).
\end{equation}
As such we have a CM-curve in $\PP^2$ with Hilbert polynomial $3t+1$.  
We will also view $\PP^2$ as a closed subscheme in $\PP^3$ given by the 
homogeneous ideal $(z)\subset k[x,y,z,w]$. The finite morphism 
\eqref{map_to_P2} composed with the closed immersion gives a finite morphism
\begin{equation}\label{map_to_P3}
i \colon X\ra \PP^3,
\end{equation}
and thus a CM-curve in $\PP^3$.

\begin{lemma}\label{degenerate} 
Let $X\subseteq \PP^3$ be the closed subscheme 
given by the graded ideal 
$I=(u^2, uy-x^2,xu)\subseteq k[x,y,u,w]$, and let $i\colon X\ra \PP^3$ be 
the finite morphism \eqref{map_to_P3}. The completion  of the local ring of the point  
$(X,i)$ in $\CM^{3t+1}_{\PP^3}$ is the power series ring
$R=k[[a_1, \ldots, a_8, b_9, \ldots , b_{12}]]$. The universal family over 
$\Spec(R)$ is given by the pair 
$(X_R, i_R)$, where $X_R$ is given by the 
homogeneous ideal $J\subseteq R[x,y,u,w]$ generated by the maximal minors of
\begin{equation}
\label{u_matrix}
\begin{pmatrix}
 x + a_2 w& a_7 y + a_6 w&  u  \\
 y +a_1w& u + a_5 x+ a_4 y + a_3 w & x +a_8 w
\end{pmatrix},
\end{equation}
and the morphism $i_R\colon X_R\ra \PP^3_R=\Proj (R[x,y,z,w])$ is determined by the linear map
$$i_R^*(x,y,z,w)=(x,y,b_9x+b_{10}y+b_{11}w+b_{12}u,w).$$
\end{lemma}

\begin{proof} 
We need to describe all deformations of the pair $(X,i)$.
By Proposition \ref{regular} we may assume for
a deformation $(X_R,i_R)$ over a local ring $R$ that
the scheme $X_R$ is embedded in some  $\PP^3_R$ and the morphism $i_R$ is induced by a linear map. 
We start with infinitesimal deformations.
The embedded deformations of the curve 
$X\subset \PP^3=\Proj(k[x,y,u,w])$ are described in 
Lemma \ref{defo_curve}, and are given by the maximal minors of 
\[
\begin{pmatrix}
 x + a_2 w& a_7 y + a_6 w& u+a_{12}x+a_{11}y+a_{10}u+a_9 w\\
 y +a_1w& u + a_5 x+ a_4 y + a_3 w & x +a_8 w
\end{pmatrix},
\]
over the polynomial ring $k[a_1, \ldots, a_{12}]$. 
The morphism $i\colon X\ra \PP^3$ is determined by the  
linear map $i^*(x,y,z,w)=(x,y,0,w)$.
This map  is deformed  
 by perturbing all possible entries. 
We need 16 variables $b_1, \ldots, b_{16}$ to perturb every component as a 
linear expression in $x,y, u,w$.

As we are allowed to perform coordinate transformations
in the source of $i$, we can use the invertible map sending $x$ to the first component of the deformed map,  $y$ to the second,  $u$ to $u$ and $w$ to the fourth component.  By performing these transformations the  map
 simplifies  to the linear map $i_R^*$ determined by
\[
i_R^*(x,y,z,w)=(x,y,b_9x+b_{10}y+b_{11}w+b_{12}u,w)\;.
\]
The matrix displayed above, describing the embedded deformations of the curve, then changes.
But, after a linear change in the $a_i$'s, depending on the $b_j$'s we can take this matrix to be of the same form.
 The remaining available transformations can be used to
make the $(1,3)$-entry of the matrix displayed into $u$. After another linear
change of coordinates in the $a_i$'s the infinitesimal deformations are given by
the matrix \eqref{u_matrix} and the morphism $i_R$. These formulas define
a deformation over the  completion of the 
polynomial ring $k[a_1, \ldots, a_8, b_9, \ldots, b_{12}]$ in the ideal $(a_1, \ldots, a_8, b_9, \ldots b_{12})$.

\end{proof}

The next two lemmas will be used to describe the space of $\CM$ curves, with polynomial $3t+1$, in the plane
and in $\PP^n$ with $n>3$ respectively. 

\begin{lemma}\label{P2 degenerate} 
Let $X\subseteq \PP^3$ be the closed subscheme given by the 
graded ideal $I=(u^2, uy-x^2,xu)\subseteq k[x,y,u,w]$, and let 
$i_1\colon X\ra \PP^2$
be the morphism \eqref{map_to_P2}. The completion of the local ring of the point 
$(X,i_1)$ in $\CM^{3t+1}_{\PP^2}$ is given by the power series  ring 
$R_1=k[[a_1, \ldots, a_8]]$. The universal family over $\Spec(R_1)$ 
is given by the pair 
$(X_{R_1}, i_1)$, where $X_{R_1}$ is given by the ideal $J\subseteq R_1[x,y,u,w]$ 
generated by the maximal minors of \eqref{u_matrix}
and the morphism $i_{R_1}\colon X_{R_1}  \ra \PP^2_{R_1},$ is the one obtained from $i_1$. 
\end{lemma}

\begin{proof} 
We proceed as in the proof of  Lemma \ref{degenerate}.
Now we perturb the map $i_1$ in all possible ways, and we can use 
coordinate transformations in the source to undo these perturbations.
\end{proof}

\begin{lemma}\label{Pn degenerate} 
Let  $X\subseteq \PP^3$ be the closed subscheme given by the 
graded ideal $I=(u^2, uy-x^2,xu)\subseteq k[x,y,u,w]$ as in Lemma \ref{P2 degenerate},  
but now let $i_2\colon X\ra \PP^n=\Proj (k[x,y,w, z_1, \ldots, z_{n-2}])$ $(n>3)$
be the morphism  determined by the linear map that sends $z_i$ to $0$ $($for $i=1, \ldots, n-2)$ and is the identity on $x,y $ and $w$.

The completion of the local ring of the point 
$(X,i_2)$ in $\CM^{3t+1}_{\PP^n}$ is given by the power series ring 
\[R_2=k[[a_1, \ldots, a_8,b_{1,1},\dots,b_{1,4}, b_{2,1}, \dots, b_{n-2,4}]].\]
The universal family over $\Spec(R_2)$ is given by the pair 
$(X_{R_2}, j_2)$, where $X_{R_2}$ is given by  the maximal minors  of \eqref{u_matrix} 
and the morphism $i_{R_2}\colon X_{R_2} \ra \PP^n$ is determined by the linear map that is the identity on $x,y$ and $w$, and 
that sends 
$(z_1, \ldots, z_{n-2})$ to 
\[(b_{1,1}x+b_{1,2}y+b_{1,3}w+b_{1,4}u,\dots,
b_{n-2,1}x+b_{n-2,2}y+b_{n-2,3}w+b_{n-2,4}u).
\]
\end{lemma}

\begin{proof} 
Similar to the  proof of  Lemma \ref{degenerate}.
\end{proof}

\begin{prop}\label{smooth}
The space $\CM^{3t+1}_{\PP^3}$ is smooth  and irreducible of dimension 12.
\end{prop}

\begin{proof} 
Let $(X,i)$ be a $k$-point of  $\CM =\CM^{3t+1}_{\PP^3}$
 By Proposition \ref{regular} we may assume that $X$
is embedded in some $\PP^3$ and that the morphism $i$ is the restriction of a rational map $h\colon \PP^3 \dashrightarrow \PP^3$. 
In particular,  $X$ itself is a curve on the twisted cubic component. 
Furthermore, the same holds for all infinitesimal deformations. 

If $i\colon X\ra \PP^3$ is a closed immersion, then we
may assume that $h$ is the identity and remains so
under deformation. Consequently the infinitesimal deformations  of $i\colon X\subset \PP^3$ are the embedded deformations of $X$.
Therefore the tangent space at the point $(X,i)$
is isomorphic to the tangent space to the twisted cubic component in the point $X$, which has dimension 12
\cite[p. 766]{PieneSchlessinger}.

Assume now that $i\colon X \ra \PP^3$ is not a closed immersion. By perturbing the coefficients defining the rational map $h$ we can deform the morphism $i\colon X\ra \PP^3$ to a closed immersion.  Thus any point on $\CM$ is the specialization of a twisted cubic in $\PP^3$. It follows that the space $\CM$ is irreducible, and consequently to prove the proposition it suffices to show that the tangent space at any point  has dimension 12.

As the image $i(X)=D$ spans a plane, the  map $h$
is the projection from a point $Q$. Because $i$ fails to be an isomorphism over exactly one point of $D$, the point $Q$ does not lie on two different tangent or secant lines.
Considered as a point on the twisted cubic component, the curve $X\subset \PP^3$ cannot be the most degenerate curve, 
which is a triple line given by the square of the ideal of a line \cite[Section 1.b]{Harris}. Such a curve has everywhere embedding dimension 3, so is not locally planar. Every
other Cohen-Macaulay curve on the twisted cubic component degenerates to triple line on a quadric cone. For such a curve the center
$Q$ of the projection does not lie on the tangent plane
to the cone containing the line. All maps $i\colon X \ra \PP^3$ with $X$ a triple line on a cone and $i$ generically an isomorphism
are projectively equivalent, so it suffices to compute the tangent space in one specific example, for which 
we take the curve $(X,i)$ of Lemma \ref{degenerate}.
By specialization the dimension of the tangent space at any point is bounded from above by the dimension of the tangent space at this $(X,i)$, which is 12.
\end{proof}

\begin{cor}\label{n-smooth} 
For $n>3$, the space $\CM^{3t+1}_{\PP^n}$ is smooth and irreducible of dimension $4n$.
\end{cor}
\begin{proof} Let $i\colon X\ra \PP^n$ be a CM-curve.   If $i$ is a closed immersion it follows 
from Lemma \ref{embedded} that the curve $X$ is contained in some $\PP^3$, and from Proposition \ref{regular} that 
infinitesimal deformations of $i(X)\subset  \PP^n$ are the embedded deformations of $X$. The dimension of the tangent space is the dimension of $H^0(X,\calN_{X/\PP^n})$, which can be computed as the sum of the dimensions of
$H^0(X,\calN_{X/\PP^3})$ and $H^0(X,\calN_{\PP^3/\PP^n}\otimes\calO_X)$, where $\calN$ stands for the normal sheaf. The dimension is $12+4(n-3)=4n$.

If $i$ is not a closed immersion then we have by Proposition \ref{not immersion} that the image is planar. 
Arguing  as in the proof of Proposition \ref{smooth} we see
that the space $\CM$ is irreducible and that to conclude the proof it suffices to compute the dimension of the tangent space in  the explicit example of  Lemma \ref{Pn degenerate}. There the dimension is $4n$.
\end{proof}

 \begin{rem} We note that 
the space $\CM^{3t+1}_{\PP^n}$ is smooth and irreducible of dimension $4n$, for all $n\ge 3$.   The case $n=2$ will be described quite explicitly in Section \ref{sec:singular sections cubics}. 
In particular, the space $\CM^{3t+1}_{\PP^2}$ is smooth of dimension 8. 
This can also
be shown using Lemma \ref{P2 degenerate}.
 \end{rem}
 
\begin{rem}
For CM-curves $(X,i)$ with Hilbert polynomial $2t+2$ the analogue of Proposition \ref{not immersion} holds: if the map $i$ is not a closed immersion, then the image
$D=i(X)$ is a singular plane conic, 
and there is exactly one point $P\in D$ over which the morphism $i$ fails to be an isomorphism. If $i\colon X\ra \PP^3$ is a closed immersion, then the curve is 
not arithmetically Cohen-Macaulay, but one can still show that  
$\CM^{2t+2}_{\PP^3}$ is smooth and irreducible of dimension $8$. 
\end{rem}

\section{Plain double points}

In this section we focus on CM-curves where the non-isomorphism locus is the simplest possible. The following definition is motivated by Proposition \ref{not immersion}.
\begin{defn}\label{simple}Let $i\colon X \ra \PP^n$ be a 
CM-curve, and let $D\subseteq \PP^n$ denote its schematic image. A point $P$ in $D$ is a \textit{plain double point}  if the 
$\calO_{D,P}$-module $\calK\otimes_{\calO_D}{\calO_{D,P}}$ has length one, where $\calK=i_*{\calO}_X/\calO_D$.
\end{defn}

\begin{rem} The definition of a plain double point requires  the double point locus 
of a map $i\colon X \ra \PP^n$ to be as simple as possible. The point $P$ is always a singular point of the image, as a birational morphism
onto a smooth curve is an isomorphism (apply this argument to a suitable affine neighbourhood of $P$). 
But the singularity  might be of higher type. 
For instance, let $X$ be three lines in $\PP^3$ meeting in one point,
not lying in a plane,
and let 
$i\colon X \ra \PP^2$ be a general projection. Then the image $D=i(X)$ consists of three 
lines in the plane through one point $P$. The point $P\in D$ is a plain double point of the CM-curve
$i\colon X \ra \PP^2$. However the singularity $P \in D$ is a triple point and not 
a (planar) double point.
\end{rem}

\subsection{The condition (R.C.)} In the following the Ring Condition (R.C.) of De Jong and Van Straten 
\cite{DucoTheo} plays an important role. We recall the setup. Let $(R,\m)$ be a local ring, let $A$ be a local $R$-algebra, flat and Gorenstein over $R$.  An $A$-module $M$ is Cohen-Macaulay over $R$ if it is flat and $\overline{M}=M\otimes_RR/\m$ is Cohen-Macaulay. Write $\overline A$ for $A\otimes_RR/\m$. If  furthermore the codimension $\dim \overline{A}- \dim_{\overline{A}}\overline{M}$  
is zero we say that $M$ is maximally Cohen-Macaulay (MCM) over $R$.  A fractional ideal is a finitely generated sub-module $M$ of the total fraction ring $QA$ of $A$, containing a non-zero divisor.
By \cite[Propopsition 1.7]{DucoTheo}
the duality functor $M\mapsto \Hom_A(M,A)$ is an inclusion reversing involution on the category of fractional MCM's over $R$,
commuting with specialization for MCM's.
In this situation we have the following result \cite[Propopsition 1.8]{DucoTheo}.
\begin{prop}\label{DucoTheo}
Let $A$ be an $R$-algebra, flat and Gorenstein over $R$,  let $B$ be a fractional MCM $A$-module over $R$ and let $C=\Hom_A(B,A)$  be its dual module. Then $B$ is a ring
$($with ring structure induced from $B\subset QA)$ if and only if 
the natural inclusion map 
\begin{equation}\label{RC}
\Hom_A(C,C)\hookrightarrow\Hom_A(C, A)
\end{equation}
is an isomorphism.
\end{prop}

A submodule $C$ of a ring $A$, that is an ideal, satisfies the Condition (R.C.)
if the natural inclusion map \eqref{RC} is an isomorphism. 

\begin{prop}\label{sing section}  Let $(X_R,i_R)\in \CM_{\PP^2}(\Spec R)$ with $R$ a local ring. 
Let $i\colon X\ra \PP^2$ denote the induced curve over the closed point in $R$, and assume that $P\in \PP^2$ is a plain double point. 
Let $\xi \in \PP^2_R$ be the image of the point $P$ under the natural inclusion $\PP^2\subseteq \PP^2_R$. Then $\xi$ is a point on the schematic image $D_R\subseteq \PP^2_R$ of $X_R$, and we have a presentation
\begin{equation}\label{pres_R}
\oplus_{i=1}^2 \calO_{\PP^2_R,\xi} \maplongright{\begin{pmatrix} 
g & f\\ s & t \end{pmatrix}}  \oplus_{i=1}^2 
\calO_{\PP^2_R,\xi} \ra  i_{R*}{\calO}_{X_R,\xi} \ra 0.
\end{equation}
where  $s$ and $t$  are elements in  the maximal ideal of $\calO_{\PP^2_R,\xi}$
and $g$ and $f$ lie in the ideal $\n_R=(s,t)$. 
Furthermore, we have the equality of ideals $\n_R =\Ann_{D_R}(i_{R*}\calO_{X_R}/\calO_{D_R})$.
\end{prop}

\begin{proof} 
Let $\n$  denote the maximal ideal of the local ring 
$A=\calO_{D,P}$, and set $B=i_*\calO_{X,P}$.  It follows that $\n$ equals the annihilator ideal 
$\Ann_A(B/A)$, which equals the conductor ideal 
$\{a\in A\mid aB \subseteq A\}$. The
fact that $B$ is a ring gives by Proposition \ref{DucoTheo}
that $B=\Hom_A(\n,A) $ is isomorphic to $\Hom_A(\n,\n)$.

We may assume that the plain double point $P$ is $(0:0:1)$, so the maximal ideal $\n$ is the ideal $(x,y)$.  As $P$ is a plain double point
the $A$-module $B$ is generated by two  elements, say $1$ and 
$u$. The element $u$ corresponds to an endomorphism $\mu_u\colon \n \ra \n$. We get that $\mu_u(x)=-\bar g$ and $\mu_u(y)=-\bar f$ for some $\bar f$ and 
$\bar g$ in $\n$. In other words we have that $ux+\bar g=0$ and $uy+\bar f=0$. We get a surjective map $A\oplus A \ra B$ given by
$(a_1,a_2)\mapsto a_1+a_2u$. The kernel contains the elements
$(\bar g,x)$ and $(\bar f, y)$ and they in fact generate the kernel: the image of $(a,1)$ is $a+u$, which is not zero as $u \notin A$.
We have therefore the presentation
\begin{equation}\label{pres_loc}
\oplus_{i=1}^2 A \maplongright{\begin{pmatrix} 
\bar g & \bar f\\ x & y \end{pmatrix}}  \oplus_{i=1}^2 
A \ra B \ra 0.
\end{equation}
As $B$ is a Cohen-Macaulay ${\calO}_{\PP^2,P}$-module of codimension one,
generated by $1$ and $u$, it has a free resolution of length one. Therefore we have a presentation  over ${\calO}_{\PP^2,P}$, obviously given by the same matrix 
where we use the same notation for elements in ${\calO}_{\PP^2,P}$ as for elements in $A$. The determinant
$x\bar f -y\bar g$ of the matrix defines the scheme-theoretic image 
$i(X)=D$  locally around $P$.
By assumption $X_R\ra \Spec(R)$ is flat and it follows that the presentation over $\calO_{\PP^2,P}$  can be lifted to a presentation over 
$\calO_{\PP^2_R,\xi}$, where $\xi$ is the image of $P$,  see \emph{e.g.} \cite{ArtinTata}.
Thus, there are elements $s, t, g,f$ 
 in $\calO_{\PP^2_R,\xi}$ that 
specialize to $x,y,\bar g, \bar f$ in $\calO_{\PP^2, P}$, respectively, giving us the exact sequence
\begin{equation}\label{matrix1}
\oplus_{i=1}^2 \calO_{\PP^2_R,\xi} \maplongright{\begin{pmatrix} 
g & f\\ s & t \end{pmatrix}}  \oplus_{i=1}^2 
\calO_{\PP^2_R,\xi} \ra  i_{R*}\calO_{X_R,\xi}\ra 0.
\end{equation}

The ideal $I$ of the schematic image $D_{R}\subseteq \PP^2_R$
in $\calO_{\PP^2_R,\xi}$ is given by the  kernel of the ring homomorphism $\calO_{\PP^2_R,\xi}\ra \calO_{X_R,\xi}$. It contains 
$\Fitt^0_{}(i_{R*}{\calO}_{X_R,\xi})$, which is the
principal ideal $(sf-gt)$. As the map $i_R$ is generically an embedding,
any other element in the kernel would have to be supported at $\xi$, but as the ideal $(sf-gt)$ has no embedded
components the equality of ideals $I=(sf-tg)$ follows.
So the matrix  appearing in \eqref{matrix1}  also defines a presentation 
of $i_{R*}\calO_{X_R,\xi}$ as $\calO_{D_R,\xi}$-module
and the ideal $\n_R= (s,t)$ is the annihilator ideal $\n_R=\Ann_{D_R}(i_{R*}\calO_{X_R,\xi}/\calO_{D_R,\xi})$. Furthermore, as the elements $s$ and $t$ specialize to $x$ and $y$ we have that $\calO_{\PP^2,\xi}/\n_R=R$. Thus $\n_R$ is flat. By the condition
(R.C.) the  element $u\in i_{R*}\calO_{X_R,\xi}$ corresponds to an endomorphism $\mu_u:\n_R\ra\n_R$, with $\mu_u(s)=-g$ and $\mu_u(t)=-f$.  This means 
that $g,f\in \n_R$. 
\end{proof}

\begin{cor}\label{properties} 
Let $I$ be the defining ideal of $\Spec(\calO_{D_R,\xi})$ in $\Spec({\calO}_{\PP^2_R,\xi})$.
Then we have that
\begin{enumerate}
\item The ideal $I=\Fitt^0_{}(i_{R*}{\calO}_{X_R,\xi})$. In particular, the ideal $I$ is  principal and we have an inclusion of ideals $I \subseteq \n_R^2$.
\item The $R$-algebra $\calO_{D_R,\xi}$ is flat.
\item The scheme defined by the annihilator ideal $\Ann (i_{R*}\calO_{X_R}/\calO_{D_R})_{\xi}$ in  $\Spec(\calO_{D_R,\xi})$ is isomorphic to $\Spec R$.
\end{enumerate}
\end{cor}

\begin{proof} All three assertions are established in the proof of the above proposition. 
\end{proof}

\begin{rem} 
The Ring Condition (R.C.) for the ideal $\n$ or $\n_R$
is equivalent to the condition that the entries  of the first row of the matrix in \eqref{pres_loc} or \eqref{pres_R}
lie in the ideal $\n$, respectively $\n_R$. This is the content in this 
special case of Catanese's \textit{Rank Condition} (R.C.) \cite{Catanese};
for the terminology see also \cite[Remark 1.13]{DucoTheo}.
\end{rem}

\section{Singular sections of cubics}\label{sec:singular sections cubics}

We give in this section  an explicit description of the space of CM-curves in the 
plane having Hilbert polynomial $3t+1$. The space of such CM-curves is identified 
with plane cubics together with a singular section.

\subsection{Critical locus}  Let $\varphi \colon X \ra S$ be a flat morphism of schemes, of pure relative  dimension $d$. The critical locus of the morphism $\varphi$ is the closed subscheme $C(\varphi)\subseteq X$ given by the ideal sheaf $\Fitt^d(\Omega^1_{X/S})$, where $\Omega^1_{X/S}$ is the sheaf of differentials, see \cite{Teissier}. A section $\sigma \colon S \ra X$ of the morphism $\varphi \colon X\ra S$ is a {\em singular section} if it factorizes through the critical locus $C(\varphi)$. 
Thus, in the commutative diagram 
\[
\begin{tikzcd}
 C(\varphi) \arrow[r,hook] & X \arrow[d, "\varphi"] \\ 
S \arrow[u,dashrightarrow] \arrow[r,"="] \arrow[ur,"\sigma"]& S
\end{tikzcd}
\]
the singular section is represented by a dashed arrow. 

\begin{ex}\label{Kahler diff} Let $R$ be a ring, and let $f\in R[x_1, \ldots, x_n]$ be an element of the polynomial ring defining a flat family $X$ of hypersurfaces over $\Spec(R)=S$. One then computes, see \emph{e.g.} \cite[Example 1, p. 588]{Teissier} that the $(n-1)^\mathrm{th}$ Fitting ideal of $\Omega_{X/S}$ is generated by the partial derivatives of $f$. Thus we have that the ideal of the critical locus is
\[ 
\left(f, \frac{\partial f}{\partial x_1}, \ldots, \frac{\partial f}{\partial x_n}\right)\subseteq R[x_0, \ldots, x_n].
\] 
\end{ex}

\subsection{Singular cubics} Let $SC$ denote the functor parametrising cubics in $\PP^2$ with a singular section. That is, the $S$-valued points of $SC$ are pairs $(D,\sigma)$ where $D\subseteq \PP^2_S=\PP^2\times S$ is a flat family of cubics over $S$, and where $\sigma \colon S\ra D$ is a singular section.

\begin{lemma} The critical locus $C(\varphi)$ of the universal family of cubics 
$\varphi\colon Z\subseteq \PP^2\times \operatorname{Hilb}^{3t}_{\PP^2} \ra \operatorname{Hilb}^{3t}_{\PP^2}$
represents the functor $SC$. The scheme $C(\varphi)$ is smooth, projective and of dimension 8.
\end{lemma}

\begin{proof} Let $S$ be a scheme, and let $s\colon S\ra C(\varphi)$ be a morphism. The morphism $s$ is determined by a pair $(\sigma, t)$, where $\sigma \colon S\ra \PP^2$ and $t\colon S\ra \operatorname{Hilb}^{3t}_{\PP^2}$ are morphisms that together factorize through the closed subscheme 
$C(\varphi)\subseteq \PP^2\times \operatorname{Hilb}^{3t}_{\PP^2}$. The morphism $t\colon S\ra \operatorname{Hilb}^{3t}_{\PP^2}$ is equivalent with having a cubic $D_S\subseteq \PP^2_S=\PP^2\times S$, flat over $S$. The morphism $\sigma \colon S \ra \PP^2$ is the same as having a section of $\PP^2_S \ra S$. Now, as our pair $(\sigma, t)$ is a point of the critical locus $C(\varphi)$ it means that the partial derivatives of the cubic vanish over $\sigma$. In other words the section $\sigma \colon S\ra \PP^2_S$ factors through the critical locus of the cubic $D_S\ra S$. And conversely, given a flat family $D\subseteq \PP^2_S$  of cubics over $S$ with singular section
$\sigma \colon S\ra D\subseteq \PP^2_S$ we obtain  morphisms  $\sigma \colon S\ra \PP^2$ and  $t\colon S\ra \operatorname{Hilb}^{3t}_{\PP^2}$ 
that together factorize through $C(\varphi)$.

From  Example \ref{Kahler diff} we get that the ideal of the critical locus is locally defined by the partial derivatives of the cubic. It follows from local calculations that $C(\varphi)$ is smooth of dimension $11-3=8$. 
\end{proof}

\begin{prop}\label{CM i P2}The functor $\CM^{3t+1}_{\PP^2}$ of CM-curves in $\PP^2$ having Hilbert polynomial $3t+1$  is isomorphic to the functor $SC$ of cubics 
with singular section. In particular $\CM^{3t+1}_{\PP^2}$ is represented by the scheme $C(\varphi)$, given as the critical locus of the universal family of cubics in the plane.
\end{prop}

\begin{proof}
First we define a morphism $\phi\colon \CM^{3t+1}_{\PP^2} \ra SC$. Given a scheme $S$ and $i_S\colon X_S\ra \PP^2_S$ an $S$-valued point of $\CM^{3t+1}_{\PP^2}$, we let $D_S\subseteq \PP^2_S$ be the schematic image of $X_S$ and $N_S\subseteq D_S$ the subscheme defined by the annihilator of $i_{S*}\calO_{X_S}/\calO_{D_S}$. Let $R$ be the local ring of a closed point of $S$
and denote by $i\colon X\ra \PP^2$ the curve over this closed point. By Proposition \ref{not immersion} the image $D=i(X)$ is a cubic curve and there is one unique point $P\in D$ where the induced map $i\colon X\ra D$ is not an isomorphism. The point $P$ is a plain double point. 
By Corollary \ref{properties} we have that the schematic image $D_R\subseteq \PP^2_R$ is flat over $R$, and that the subscheme $N_R\subseteq D_R$ defined by the annihilator of $i_{R*}\calO_{X_R}/\calO_{D_R}$, determines a singular section of $D_R\ra \Spec(R)$.  
Thus $(D_S, N_S)$ is a flat family of cubics in $\PP^2$  with a singular section.

Next we define a morphism $\theta \colon SC \ra \CM^{3t+1}_{\PP^2}$. Let $(D_S, N_S)$ be a flat family of cubics in $\PP^2$  with a singular section. We have that $D_S$ is given by a  cubic form $Q \in \calO_S[x,y,w]$. Consider an open affine set $U\subset S$ on which the section $N_U$ is given by two linear independent forms $s$ and $t$.  As the section is singular we have that the cubic  form $Q\in  \calO_U[x,y,w]$ can be written as
$Q=s^2f_1+st(f_2-g_1)-t^2g_2$.
Consider now a matrix factorisation of $Q$
\begin{equation}\label{matrix M}
M=\begin{pmatrix} 
g_1 s +g_2 t & f_1 s+f_2t\\
s & t
\end{pmatrix}.
\end{equation}
We can view the matrix $M$ as the presentation matrix of a sheaf $\calF_U$ on $\PP^2_U$. That is, we have the global presentation
\[
\oplus_{i=1}^2
\calO_{\PP^2_U}(-2)  \mapright M
\calO_{\PP^2_U}\oplus \calO_{\PP^2_U}(-1) 
\ra \calF_U \to 0.
\]
With generators $1,u$ of $\calF_U$ we have the
relations  $(u+g_1)s+g_2t=0$ and $f_1s+(u+f_2)t=0$.
By formally eliminating $s$ and $t$ we obtain a third relation
$(u+f_2)(u+g_1)=g_2f_1$. 
These relations can be written as the maximal minors of the matrix
\begin{equation}\label{matrix minors}
\begin{pmatrix} s & -g_2 & u+f_2 \\ t & u+g_1 & -f_1 \end{pmatrix}.
\end{equation}
The maximal minors  define an  arithmetically
Cohen-Macaulay curve $X_U\subset \PP^3_U=U\times \Proj (k[x,y,u,w])$, flat over $U$.  For each point in $U$, the fiber is a curve with Hilbert polynomial  $3t+1$, and the curve does not pass through the point  $(0:0:1:0)$. Projection from  $(0:0:1:0)$ induces 
a map $i_U\colon X_U \ra \PP^2_U$, 
that is an isomorphism onto its image $D_U$ outside the section $N_U$.

A different choice of the sections
$s$ and $t$ and of the matrix factorization gives a curve 
isomorphic to $X_U$.
The image of $X_U$ in $\PP^2_U$ and the section $N_U$ are independent of these choices,
and it follows that the curves defined for different open sets $U\subseteq S$ glue together to
a curve $i_S\colon X_S \ra \PP^2_S$, and thus an $S$-valued point  of $\CM^{3t+1}_{\PP^2}$.

 We next  show  that the two morphisms constructed are inverse of each other. For this we may assume that the base is the spectrum of a local ring $R$.
 Let $(D_R, N_R)$ be a flat family of cubics with a singular section. Let $i_R \colon X_R \ra \PP^2_R$ be the CM-curve  
 we get by applying the morphism $\theta$ to the pair $(D_R, N_R)$.  From the construction we have that $\phi \circ \theta$ is the identity, and we verify that $\theta \circ \phi$ is the identity. 
 Outside the section $N_R$ the two curves $X_R$ and $D_R$ are isomorphic, so in particular $X_R \setminus i_R^{-1}(N_R)$ is determined by the pair $(D_R,N_R)$.  
 Over the closed fiber we have that the non-isomorphism locus is one point $P\in \PP^2$, and we let $\xi \in \PP^2_R$ be the image of $P$ under the natural inclusion. 
 By Proposition \ref{sing section} we have that the ideal $\n_R$ of the section $N_R \subset D_R$ is contained in the ideal defining $\xi$. 
 Let $A=\calO_{R,\xi}$ and $B=(i_{R*}\calO_{X_R})_{\xi}$. The last statement of Proposition \ref{sing section} tells us  that the ideal $\n_{R}$ is the annihilator ideal $\operatorname{Ann}_{A}(B/A)$. 
 Hence $\n_R$ is the dual module $\operatorname{Hom}_A(B,A)$. By Proposition \ref{DucoTheo} we then have that the $A$-algebra $B$ is the endomorphism ring $\operatorname{Hom}_A(\n_R, \n_R)$.
 Thus the structure sheaf $i_{R*}\calO_{X_R}$ is determined, up to isomorphism, by the pair $(D_R,N_R)$. It follows that  $\theta \circ \phi$ is the identity.
 \end{proof}

\begin{rem} The proof could have been shortened using Zariski Main Theorem.  We have chosen to give an explicit proof that we believe is more illuminating.
\end{rem}

\begin{rem}There is a natural map from $C(\varphi)$ to $\operatorname{Hilb}^{3t}_{\PP^2}=\PP^9$, obtained by composing the closed immersion with the projection. This is a map from the space of cubics with a singular section $SC$ to the  Hilbert scheme forgetting the singular section. The image will be the set of singular cubics in the plane.  Outside the locus of cubics with multiple components the map is the normalization morphism, see \cite[Theorem 5.5.1]{Teissier}. Over the set corresponding to multiple components the map is a proper modification.
\end{rem}

\subsection{Remarks on conics} It turns out that the situation with CM-curves in $\PP^2$ having Hilbert polynomial $2t+2$ can be treated quite analogously to the situation with twisted cubics. For readability we did not merge those similar arguments in the above text. We claim that they show that the space $\CM^{2t+2}_{\PP^2}$ of CM-curves in the plane having Hilbert polynomial $2t+2$ is isomorphic to the space of plane conics with a singular section.

\section{The space of twisted cubics}
This section contains our main contribution that identifies 
the space of CM-curves in $\PP^3$ having Hilbert polynomial $3t+1$ with one 
specific component of the Hilbert scheme of twisted cubics in $\PP^3$.
We also treat the case $n>3$.

\begin{lemma}\label{Fittingflat}
Let $i_R\colon X_R\ra \PP^3_R$ be an element of $\CM_{\PP^3}^{3t+1}(\Spec(R))$, where $R$ is a local ring. Suppose that the induced map $i\colon X \ra \PP^3$ over the closed
point in $R$ is not a closed immersion. Then the closed subscheme $Z\subseteq \PP^3_R$ defined by the $0^\mathrm{th}$ Fitting ideal sheaf
of $i_{R*}{\calO}_{X_R}$ is flat over $\Spec(R)$, with Hilbert polynomial $3t+1$.
\end{lemma}

\begin{proof} The image of the CM-curve $i\colon X\ra \PP^3$ over the closed fiber is by Proposition \ref{not immersion} a plane cubic. We may assume that the image lies in the plane $\Pi=\{z=0\} \subset \PP^3 =\Proj (k[x,y,z,w])$ and that the
plain double point is $(0:0:0:1)$. 
From Proposition \ref{regular} it follows that we may assume that the curve $X$ lies in $\PP^3 =\Proj(k[x,y,u,w])$ and that the morphism $i$ is induced by the linear map that sends $(x,y,z,w)$ to  $(x,y,0,w)$. In particular we may assume that the curve 
$X$ does not pass through the point $(0:0:1:0)$. 
We compose the map $i_R\colon X_R\ra \PP^3_R$ with the rational projection  $\PP_R^3\dashrightarrow \Pi_R$. 
This
provides us with an element of $\CM^{3t+1}_{\Pi}(R)$ and the
image of $X_R$ is by Proposition \ref{CM i P2} a cubic $D_R\subset \Pi_R$ together with a singular section $N_R$. After a change of coordinates we may assume that the section is given by $s=x$ and $t=y$. 
Then the maximal minors of \eqref{matrix minors}   (with $f_2=0$) that generate the ideal  of $X_R$ in $\PP^3 =\Proj(k[x,y,u,w])$ are
\[ (xu+g, yu+f, u(u+g_1)-f_1g_2), \]
where $g=g_1x+g_2y$ and $f=xf_1$. The cubic $D_R\subset \Pi_R$ is given by $Q=f_1x^2-g_1xy-g_2y^2 \in R[x,y,w]$.
The map $i_R\colon X_R\ra \PP^3_R$ is then  obtained from the linear map $i^*$ that sends $z$ to $\alpha w+\beta u+ \gamma y + \delta x$
with $\alpha$, $\beta$, $\gamma$ and $\delta$   in the maximal ideal of $R$. By a coordinate change in the image, making $z-\alpha w - \gamma y - \delta x$ the new $z$-coordinate we bring the map into the
final form: $i^*$ sends $z$ to $\beta u$ and is the identity on $x,y$ and $w$.

We now compute a presentation of $X_R$ over $\PP^3_R$. 
 In \eqref{matrix M} we have a presentation of $X_R$ over $\PP^2_R$. We identify the plane with the subscheme in $\PP^3_R$ defined by $z=0$. We have that $i_*\calO_{X_R}$ is generated by  two elements $1$ and $u$ over $\calO_{\PP^3_R}$. To obtain a presentation we only need to add the relations coming from the action of $z$. We have that 
$z\cdot 1=\beta u$, and therefore   $z\cdot u=\beta u^2$. We rewrite $\beta u^2$  using the relation $u^2+ug_1-f_1g_2=0$.
Therefore we have the  presentation 
\[
\calO_{\PP^3}(-1) \oplus 3\calO_{\PP^3}(-2) \mapright M
\calO_{\PP^3}\oplus \calO_{\PP^3}(-1)  \longrightarrow 
i_*{\calO}_{X_R}\longrightarrow 0
\]
with $M$ the matrix
\begin{equation}\label{flat presentation}
\begin{pmatrix}
z & -\beta f_1g_2   &g_1x+g_2y & f_1x \\
-\beta & z+\beta g_1 & x & y
\end{pmatrix}.
\end{equation}

A computation shows that the $0^\mathrm{th}$ Fitting ideal  $I=\Fitt^0(i_*{\calO}_{X_R})$ is generated by four elements $Q,\,F_1,\,F_2,$ and $F_3$, where $Q=Q(x,y,w)$ is our recurring cubic, and where 
\begin{align*}
F_1&=z^2+\beta g_1z -\beta^2 f_1g_2\\
F_2&=zx+\beta g_1 x + \beta g_2y\\
F_3&=zy + \beta f_1x.
\end{align*}
As $R[x,y,w]/(Q)$ is a flat family of cubics the $R$-module of degree $n$ forms is locally free of rank $3n$. From the leading terms of $F_1, F_2 $ and $F_3$ it follows that the $R$-module of degree $n$ forms in the graded quotient ring $R[x,y,z,w]/(Q,F_1,F_2,F_3)$ is locally free of rank $3n+1$, adding only the free component with basis $zw^{n-1}$ to the forms determined by the cubic $Q$ alone.
Therefore $Z$ is flat over $R$.
\end{proof}

\subsection{The twisted cubic component} 
The Hilbert scheme $\Hilb_{\PP^3}$ of closed subschemes in $\PP^3$ having Hilbert 
polynomial $3t+1$ consists of two smooth components. This was proven by Piene 
and Schlessinger \cite{PieneSchlessinger} when the base field has 
characteristic different from 2 and 3, but their result is valid in any characteristic, 
see Proposition \ref{PS}. One of these components is 12 dimensional and 
contains an open subset that parametrizes twisted cubics in $\PP^3$. 
We refer to the component $\HC$ as the twisted cubic component.

\begin{thm}\label{main}
Let $\CM^{3t+1}_{\PP^3}$ be the space of CM-curves in $\PP^3$  with Hilbert polynomial $3t+1$, and let $\HC \subset \Hilb_{\PP^3}$ denote the twisted cubic component of the Hilbert scheme. We have an isomorphism 
\[
\phi \colon \CM^{3t+1}_{\PP^3} \ra \HC,
\]
mapping an $S$-valued point $(X,i)$ to the closed subscheme 
in $\PP^3_S$ determined by the $0^\mathrm{th}$ Fitting ideal 
$\Fitt^0(i_*{\calO}_X)$.
\end{thm}

\begin{proof} To show that the  morphism $\phi$ is well-defined   it suffices to 
check  it over an affine base scheme $S=\Spec(R)$.
 If the map $i_R\colon X_R\ra \PP^3_R$ is a closed immersion there is nothing to prove as the $0^\mathrm{th}$ Fitting ideal is the the defining ideal of the image of the closed immersion. If the map is not a closed immersion it follows from Lemma \ref{Fittingflat} that the $0^\mathrm{th}$ Fitting ideal defines a flat subscheme with appropriate Hilbert polynomial. We have therefore a morphism $\phi$ from $\CM^{3t+1}_{\PP^3}$ to the Hilbert scheme $\Hilb_{\PP^3}$. As $\CM^{3t+1}_{\PP^3}$ is irreducible it follows that the morphism factorizes through the twisted cubic component $\HC$ and gives the sought morphism.

Let $Z\subseteq \PP^3$ be a closed subscheme corresponding to a point on the 
twisted cubic component $\HC$. There is a divisor $\delta$ on $\HC$ corresponding to singular planar cubics together with a spatial  embedded point. If $Z$ corresponds to a point on $\HC$ but not on the divisor $\delta$ then the closed 
immersion $i\colon Z\ra \PP^3$ is a CM-curve. 
If $Z$ is a point on the divisor, then it is a plane cubic, lying in a plane $\Pi\subset \PP^3$, with an embedded point at a singular
point $P$ of the cubic, see (\cite[Lemma 2]{PieneSchlessinger}). 
By Proposition \ref{CM i P2} there is a (unique) CM-curve $i_\Pi\colon X\ra \Pi $
such that the image of $X$ is the singular cubic and the map $i_\Pi$ is not
an isomorphism onto its image at $P$. Let $i$ be the composed map
$i\colon X\ra \Pi \subset \PP^3$. Then the image by $\phi$ of the 
curve $(X,i)$ is $Z\subset \PP^3$. More explicitly, the scheme $Z$
is projectively equivalent to the scheme given by the ideal 
$I=(xz, yz, z^2, Q)$ where $Q=Q(x,y,w) $ is a cubic form singular at $(0:0:1)$ (see \cite[Lemma 2]{PieneSchlessinger}).  As in the proof of 
Lemma \ref{Fittingflat} the curve $X$ is given by the maximal minors
of the matrix  (\ref{matrix minors}) and we have a presentation 
(\ref{flat presentation}) with $\beta=0$. The $0^\mathrm{th}$ Fitting ideal is the 
ideal $(Q,z^2,zx,zy)$.
It follows that the morphism $\phi \colon \CM \ra \HC$ is
bijective.
Because the spaces are isomorphic on an open dense subset, the isomorphism follows from Zariski's Main Theorem
\cite[Chapter~III, \S9]{Redbook}
as both  spaces are smooth of dimension 12.
\end{proof}

\begin{rem} Freiermuth and Trautmann studied the moduli scheme of stable 
sheaves supported on cubic space curves \cite{FreiermuthTrautmann}.
In characteristic zero there exists a projective coarse moduli space $M^p_X$
for semi-stable sheaves on a smooth projective variety $X$ with a fixed
Hilbert polynomial $p$. For $p(t)=3t+1$ and $X=\PP^3$ all sheaves in
$M=M^{3t+1}_{\PP^3}$ are stable. The result of \cite{FreiermuthTrautmann}
is that the projective variety $M$ consists of two nonsingular, irreducible,
rational components $M_0$ and $M_1$ of dimension 12 and 13, intersecting
transversally in a smooth variety of dimension 11, see also Appendix \ref{app:FT}.

The component $M_0$ is isomorphic to the twisted cubic component $\HC$
of the Hilbert scheme.  The identification  
also uses Fitting ideals. If $i\colon X\ra \PP^3$ is a CM-curve, with Hilbert polynomial 
$3t+1$, then the module $i_*{\calO}_X$ is a stable sheaf supported on a cubic.
Thus, by forgetting the algebra structure of $i_*{\calO}_X$ we get that our 
morphism $\phi \colon \CM^{3t+1}_{\PP^3} \ra \HC$ factorizes through the 
moduli scheme of stable sheaves.
\end{rem}

\begin{ex} The following example shows that the $0^\mathrm{th}$ Fitting ideal does not always give a morphism from $\CM$ to the Hilbert scheme. Indeed, we show that here
the family of closed subschemes determined by the $0^\mathrm{th}$ Fitting ideal is not flat.

We start with a genus 2 curve embedded with the linear system $|5P|$, where $P$ is a 
Weierstrass point. More precisely, we look at the curve $X\subset \PP^3$ 
given by the homogeneous ideal generated by the maximal minors of
\[
\begin{pmatrix} 
x & u & y^2+w^2 
\\ y & x & u^2
\end{pmatrix}.
\]
The projection from $X$ to the plane $\PP^2$ in the coordinates $x,y, w$ is 
finite. Let another $\PP^3$ be given by the homogenous coordinate ring 
$k[x,y,z,w]$, and consider the family of maps $i_t\colon X_t\cong X \ra \PP^3$, $t\in T= \AA^1$,
determined by the linear map
sending $(x,y,z,w)$ to $(x,y, tu, w)$.

In the affine chart $\{w=1\}$ we get that $i_{t*}\calO_X$ has presentation
\[
\begin{pmatrix} 
z & 0 & -tx(y^2+1) & 0 & x^2 & -y-y^3 \\
-t & z & 0 & x^2 & -y & 0\\
0 & -t & z & -y & 0 & x
\end{pmatrix}.
\]
The $0^\mathrm{th}$ Fitting ideal is then the ideal generated by
\begin{gather*}
z^3-t^3x(y^2+1), \quad z^2x-t^2y(y^2+1),\\ zx^3-ty^2(y^2+1), 
\qquad x^5-y^3(y^2+1),\\
(yz-tx^2)x, \quad (yz-tx^2)y, \quad (yz-tx^2)z, \quad (yz-tx^2)t.
\end{gather*}
The family determined by this ideal is not flat, as $yz-tx^2$ is a $t$-torsion element. For $t\neq 0$ we have the 
generator $yz-tx^2$ in the ideal, making the three first generators on the 
last row above superfluous.

The family above gives an $T=\AA^1$-valued point $(X_T=X\times \AA^1, i_T)$ 
of $\CM^{5t-1}_{\PP^3}$. 
Then by taking the $0^\mathrm{th}$ Fitting ideal of $i_{T*}{\calO}_{X_{T}}$ we get a closed 
subscheme $Z \subseteq \PP^3_{\AA^1}$ which is not a flat family, and in particular the Hilbert polynomial of a fiber is not constant. 
\end{ex}

\subsection{Higher codimension}
The twisted cubic component $\TC_n$ of the Hilbert scheme $\Hilb_{\PP^n}$ for
$n>3$ has been described by Chung and Kiem 
\cite{ChungKiem};
their proof works in any characteristic.
The component $\TC_n$
is isomorphic to a component of the relative
Hilbert scheme of the  $\PP^3$-bundle $\PP\calU \ra 
G_4^{n+1}$, where $\calU$ is the universal rank 4 vector bundle on the 
Grassmannian $G_4^{n+1}$ \cite[Proposition 3.3]{ChungKiem}.
Chung and Kiem also describe a  morphism $\TC_n\ra M_0$ to a component of the moduli
scheme $M^{3t+1}_{\PP^n}$ of stable sheaves and show that this morphism realises 
$\TC_n$ as the blow-up of $M_0$ along the smooth locus of 
stable sheaves with planar support \cite[Proposition 1.3]{ChungKiem}.
We have a similar result for the space of CM-curves.

We remark that the construction with the Fitting ideal does not
give a morphism from $\CM^{3t+1}_{\PP^n}$ to $\Hilb_{\PP^n}$ if $n> 3$.
Indeed, if
the image of a CM-curve $X\ra \PP^n$ is a planar curve then the $0^\mathrm{th}$
Fitting ideal gives a scheme with Hilbert polynomial $3t+n-2$.
By a coordinate 
transformation we may assume that the plane containing the image $i(X)$ is given by $z_1=\dots=z_{n-2}=0$ and that the non-isomorphism locus on the image $i(X)$ is the point $(0:\dots:0:1)$. A presentation of $i_*\calO_X$ as $\calO_{\PP^n}$-module is then given by
the matrix
\[
\begin{pmatrix} 
z_1 & 0 & \dots &z_{n-2} & 0 &g & f \\ 
0 & z_1 & \dots &0&z_{n-2}& x & y
\end{pmatrix}.
\]
The $0^\mathrm{th}$ Fitting ideal is the ideal $I=(z_i^2,z_iz_j,z_ix,z_iy,yg-xf)$ for $ 1 \leq i \leq n-2$ and 
$i<j \leq n-2$. 
There is a $\PP^{n-3}$ of $\PP^3$'s containing the plane 
$z_1=\dots=z_{n-2}=0$ and each of these $\PP^3$'s contains a subscheme
$Z$  with Hilbert polynomial $3t+1$ with the planar singular cubic as 
1-dimensional subscheme. 

\begin{prop}The twisted cubic component $\TC_n$
of the Hilbert scheme $\Hilb_{\PP^n}$, $n> 3$, is isomorphic to the blow-up of $\CM^{3t+1}_{\PP^n}$
in the locus of CM-curves with planar scheme-theoretic image.
\end{prop}

\begin{proof} 
 Let $(X_S,i_S)$ be an $S$-valued point of $\CM=\CM^{3t+1}_{\PP^n}$. By Proposition \ref{regular} we have that
$H^0(X_S, i_S^*\calO_{\PP^n_S}(1))$ is a locally free $\calO_S$-module of rank 4. Let $G=G_4^{n+1}$ be the Grassmannian of rank 4 quotients of the free module $H^0(\PP^n, \calO(1)) $ of rank $n+1$. We define
the subfunctor $\widetilde{\CM}$ of $\CM\times G$ 
by  setting, for any scheme $S$
\[ \widetilde{\CM}(S) =\{((X_S,i_S),Q_S) \mid  i^*_S \text{ factors through } Q_S \}\;.
\]
The condition that $i_S^*\colon H^0(\PP^n_S, \calO_{\PP^n_S}(1))\ra H^0(X_S, i_S^*\calO_{\PP^n_S}(1))$ factors through $H^0(\PP^n_S, \calO_{\PP^n_S}(1)) \ra Q_S$ is a closed condition on $S$. And consequently $\widetilde{\CM}$ is represented by a closed
subspace
of the product. We have that the projection $\sigma \colon \widetilde{\CM}\ra \CM$ is an isomorphism over the open set where the corresponding curves $i_S\colon X_S \ra \PP^n_S$ are closed immersions. 

To study this morphism $\sigma$ we first describe $\CM$ in a neighbourhood of a curve
$i\colon X\ra \PP^n$ whose image is a  planar curve.
We may assume that the plane is given as ${z_1=\cdots =z_{n-2}=0}$,  that $X$ is given by an ideal $I$ in $k[x,y,u,w]$, and 
 that the map $i$ is given by $i^*(x,y,z_1,\dots,z_{n-2},w)=(x,y,0,\dots ,0,w)$. 
It follows from Lemma~\ref{Pn degenerate} and by arguing as in the 
proof of Lemma \ref{Fittingflat} that we may assume
that a neighborhood of the curve $i\colon X\ra \PP^n$
is given by perturbing the ideal $I\subset k[x,y,u,w]$ and perturbing the map $i^*$  to the map $j^*$  that sends
$(x,y,z_1,\dots,z_{n-2},w)$ to $(x,y, 
b_{1,1}x+b_{1,2}y+b_{1,3}w+b_{1,4}u,\dots,
b_{n-2,1}x+b_{n-2,2}y+b_{n-2,3}w+b_{n-2,4}u,w)$.
The locus of 
CM-curves with planar scheme-theoretic image is  locally given by
the equations $$b_{1,4}=b_{2,4}= \cdots =b_{n-2,4}=0.$$

An affine chart of the Grassmannian is obtained
by choosing four global sections of $H^0(\PP^n, \calO_{\PP^n}(1))$ that form a basis of $H^0(X,i^*{\calO}_{\PP^n}(1))$. The affine chart is then the affine space representing linear maps from the remaining $n-3$ global sections to $H^0(X,i^*{\calO}_{\PP^n}(1))$. By choosing the global sections $x,y, w, z_1$ the universal family over the affine chart is then the map that sends $z_i\mapsto \lambda_{i,1}x+\lambda_{i,2}y+\lambda_{i,3}w+\lambda_{i,4}z_1$, for every $i=2, \ldots, n-2.$
The condition that the map $j^*$ factors through the quotient map  on the affine chart of the Grassmannian is  that $b_{k,m}=\lambda_{k,m}+\lambda_{k,4}b_{1,m}$ for $2\leq k\leq n-2$, $m=1,2,3$ and $b_{k,4}=\lambda_{k,4}b_{1,4}$
for $2\leq k\leq n-2$. The last $n-3$ equations show that the morphism $\sigma$ is the blow up of  the locus of 
CM-curves with planar scheme-theoretic image.

Let $\calU$ be  the universal quotient bundle over $G$. 
By \cite{ChungKiem} the Hilbert scheme $\Hilb_{\PP^n}$
is isomorphic to the relative Hilbert scheme
$\Hilb_{\PP\calU}\ra G$, with fibres isomorphic to $\Hilb_{\PP^3}$. We define a morphism from $\widetilde{\CM}$ to the Hilbert scheme $\Hilb_{\PP\calU/ G}\ra G$ by sending a pair $((X_S,i_S), Q_S)$ to the closed subscheme in $\PP\calU$ determined by the $0^\mathrm{th}$ Fitting ideal of $i_{S*}\calO_{X_S}$ over $Q_S$.

By Theorem \ref{main} the space of $\CM$ curves in each fibre is isomorphic to the twisted cubic component of the fibre 
of the Hilbert scheme $\Hilb_{\PP\calU/ G}\ra G$.
By \cite{ChungKiem} the twisted cubic component $\TC_n$
is the irreducible smooth scheme in $\Hilb_{\PP\calU}\ra G$ where the fibers are  twisted cubic components $\HC \subset \Hilb_{\PP^3}$.
\end{proof}

\subsection{The Hilbert scheme component with two skew lines} 
The Hilbert scheme $\operatorname{Hilb}^{2t+2}_{\PP^3}$ of closed subschemes in 
$\PP^3$ over a field of characteristic zero,  having Hilbert polynomial $2t+2$ 
consists of two smooth components. The smoothness of these components was 
observed in \cite{Harris}, and a proof was given by Chen, Coskun, and 
Nollet in \cite{CCN}. One of these components $H_3$ is smooth of dimension 8 
and a general point correspond to a pair of skew lines. 
We claim that with the arguments presented in the present 
article, one can prove that we have a morphism 
$\phi \colon \CM^{2t+2}_{\PP^3}\ra H_3$ sending a $S$-valued point $(X,i)$ 
to the closed subscheme in $\PP^3\times S$ defined by the $0^\mathrm{th}$ Fitting ideal 
sheaf $\Fitt^0(i_*{\calO}_X)$. The morphism is an isomorphism.

\appendix

\section{The Hilbert scheme of twisted cubics}
\label{app:PS}

The main purpose of this section is to prove the following result.

\begin{prop}[Piene-Schlessinger]\label{PS} The Hilbert scheme $\Hilb_{\PP^3}$ is the 
union of two nonsingular rational varieties $H$ and $H'$, of dimension 12 and 
15; their intersection is non-singular, transversal, and rational of dimension 11.
\end{prop}

\begin{rem} The statement about the structure of the Hilbert scheme 
for algebraically closed fields of
characteristic different from 2 and 3 is found \cite{PieneSchlessinger}. The
reason for avoiding these characteristics lies in the deformation computation 
in Section 5 of their paper; this is not needed for the results in the 
rest of the paper. Our characteristic free Lemma \ref{PS section 5} given 
below,  replaces Lemma 6 in \emph{loc. cit.}, from where the proposition 
then follows.
\end{rem}

\begin{rem} Piene and Schlessinger reduce the computation of the
local structure of the Hilbert scheme to a deformation
computation for a saturated homogeneous ideal in $k[x,y,z,w]$
by their Comparison Theorem in \cite[Section~3]{PieneSchlessinger}.

A detailed explanation of this type of deformation
computation is given in \cite{Stevens}, which furthermore contains an easier 
example of two intersecting
lines with an embedded point at the origin, relevant for the $(2t+2)$-case. 
\end{rem}

\subsection{Tangent space calculations} We start by reproving the following, 
known, explicit description of the following open affine chart of the Hilbert 
scheme $\Hilb$ of closed subschemes in $\PP^3$ with Hilbert polynomial $3t+1$, 
a result we relied on in Section \ref{section_CM}.

\begin{lemma}\label{defo_curve} Let $Z\subseteq \PP^3$ be the closed subscheme 
given by the graded ideal $I=(u^2,uy-x^2, xu)\subseteq k[x,y,u,w]$. An affine 
open chart around the corresponding point in the Hilbert scheme $\Hilb $ is 
given by the polynomial ring $A=k[a_1, \ldots, a_{12}]$. The maximal minors of
\[
\begin{pmatrix}
 x + a_2 w& a_7 y + a_6 w& u+a_{12}x+a_{11}y+a_{10}u+a_9 w\\
 y +a_1w& u + a_5 x+ a_4 y + a_3 w & x +a_8 w
\end{pmatrix}
\]
generate an ideal $I(a)\subset A[x,y,u,w]$ that determines the restriction of 
the universal family to $\Spec(A)$.
\end{lemma}

\begin{proof} The first order deformations are determined by the global 
sections of the normal sheaf $H^0(Z,N_{Z/\PP^3})$.  It is convenient
to work in the affine chart $\{w=1\}$.  
As the ideal is determinantal, we can compute as described
in \cite{Sch}.
The sections of the normal
sheaf are then generated by the following six deformations induced
by perturbing the matrix
\[
\begin{pmatrix} 
 x + \eps_{11} & \eps_{12} & u+ \eps_{13} \\ 
 y+ \eps_{21} & u + \eps_{22} & x+ \eps_{23}
 \end{pmatrix}\;.
\] 
We describe an infinitesimal deformation by its action on
the vector $(u^2,yu-x^2,xu)$ of generators of the ideal.
Written out this action reads
\begin{align*}
\textstyle \dd{ \eps_{11} } &=(0,-x,u) & \textstyle \dd{ \eps_{12} } &=(-x,0,-y) & \textstyle \dd{ \eps_{13} } &=(u,y,0) \\
\textstyle\dd{ \eps_{21} } &=(0,u,0) & \textstyle\dd{ \eps_{22} } &=(u,0,x) & \textstyle \dd{ \eps_{23} } &=(0,-x,0). 
\end{align*}
We can also multiply the generators by linear functions. We have to consider
the action on the generators modulo the ideal $(u^2,yu-x^2,xu)$.
We find that
$u \textstyle \dd { \eps_{11} } = u \textstyle \dd { \eps_{21} }
= u \textstyle \dd { \eps_{22} } = u \textstyle \dd { \eps_{23} }
 = x \textstyle \dd { \eps_{21} } = (0,0,0)$. 
The remaining actions are as follows
\begin{gather*}
 u \textstyle \dd { \eps_{13} } =-x\textstyle \dd { \eps_{11} }
 = -x \textstyle \dd { \eps_{23} }=y\textstyle \dd { \eps_{21} }
 =(0,x^2,0)\\
 x \textstyle \dd { \eps_{22} } =- u \textstyle \dd { \eps_{12} } 
 =(0,0,x^2) \quad 
 y\textstyle \dd { \eps_{22} }=-x\textstyle \dd { \eps_{12}}
 =(x^2,0,y^2)\\
  x\textstyle \dd { \eps_{13} }=-y\textstyle \dd { \eps_{23}}
 =(0,xy,0)\quad 
 -y \textstyle \dd { \eps_{12} }=(xy,0,y^2)\quad
\end{gather*}
and lastly we have that $y\textstyle \dd { \eps_{11} } = -x \textstyle \dd { \eps_{13}} +x \textstyle \dd { \eps_{22} }$.  
We therefore get 6 additional first order deformations. As deformations of determinantal varieties are unobstructed we homogenize with respect to the variable $w$ and write a 12-dimensional family by choosing appropriate representatives. With new names for the deformation variables this can be presented as the claimed maximal minors. 
\end{proof}

\begin{prop}\label{tangent of Hilb} 
Let $Z\subseteq \PP^3$ be the closed subscheme given by the graded ideal 
$I=(z^2,zx,zy,Q)\subseteq k[x,y,z,w]$, 
where $Q=Q(x,y,w)$ is a cubic form. If $Q$ is singular at $(0:0:1)$ then the 
tangent space to the Hilbert scheme at $Z$ has dimension 16. 
If $Q$ is smooth at $(0:0:1)$, then the tangent space has dimension 15.
\end{prop}

\begin{proof} Write $Q=Q(x,y,w)=xf-yg$ in $A=k[x,y,z,w]$.
We then have the free resolution
 \[
 0 \longleftarrow A/I \longleftarrow A \mapleft \varphi A(-2)^3\oplus A(-3)
 \mapleft R A(-3)^3\oplus A(-4) 
 \]
 where $\varphi=(zx,zy,z^2,Q)$ is the map given by the generators of the ideal. The relation matrix $R$ is then
 \[
 \begin{pmatrix}
 0  & z & -y & -f \\
 -z & 0 & x & g \\
 y & -x & 0 & 0 \\
 0 & 0 & 0 & z
 \end{pmatrix}\;.
 \]
Again it is convenient to work in the affine chart $\{w=1\}$. We compute deformations
as described in \cite[Section I.6]{ArtinTata} and \cite{Stevens}.
We obtain
generators for the global sections of the normal sheaf 
$N_{Z/\AA^3}$
as the syzygies of the
transpose of the relation matrix $R$, but computed modulo the ideal 
$I=(zx,zy,z^2, Q)$. 
We give the generators of the normal sheaf by their action on the 
vector $\varphi$ of generators of the ideal.  
If $Q=xf-yg$ is singular at $(0:0:1)$ then neither $f$ nor $g$ 
have non-zero constant term. 
It follows that  generators of the normal sheaf are 
 \begin{gather*} 
 (z,0,0,0),\quad (0,z,0,0), \quad (0,0,z,0), \quad (0,0,0,z),\\
 (0,0,0,x),\quad(0,0,0,y),\quad 
 (x,y,0,0),\quad 
 (g,f,0,0).
 \end{gather*}
As the degree of a perturbation can be at most that of the
element of $\varphi$, so $2,2,2,3$ respectively, and one has to compute
modulo the ideal $I$, we get  8 additional deformations  by multiplying  
the last four generators above with the variables $x$ and $y$.
It follows that the dimension of the Zariski tangent space to $\Hilb$ is 16. 
If $Q$ has linear terms, the last generator $(g,f,0,0)$ 
is not present, and $(z,0,0,0)$ is to be replaced by $(z,0,0,f)$, and $(0,z,0,0)$ 
by $(0,z,0,-g)$. Then  the dimension of the Zariski tangent space is 15. 
\end{proof}

\begin{rem}
The ideal $(zx,zy,z^2,Q)$ with $Q$ smooth at the origin gives a
plane cubic through the origin with an embedded point at the
origin, so on the curve. It is also possible to have a plane cubic and
a point in its plane but not on the cubic. Such a curve is not a small
deformation of a curve in the intersection of the two components
of the Hilbert scheme. Consider the plane cubic $z=x^3+y^3+w^3=0$
and a point at $(0:0:t:1)$.
For $t\neq0$ the homogeneous ideal is generated by
\[
zx, zy, z(z-tw), zw^2-t(x^3+y^3+w^3)
\]
but if we specialize to $t=0$ we also need the equations
$x(x^3+y^3+w^3)$ and $y(x^3+y^3+w^3)$.
The result is the non-saturated ideal
\[
zx, zy, z^2, zw^2, x(x^3+y^3+w^3), y(x^3+y^3+w^3)\;.
\]
This is a case where the Comparison Theorem of \cite{PieneSchlessinger}
does not apply.
\end{rem}

\begin{lemma}\label{PS section 5} 
The graded ideal $I=(z^2, zx, zy,x^3)\subseteq k[x,y,z,w]$  has a universal deformation
space given by the ideal \[J=(b_{12}c_{13},b_{12}c_{14},b_{12}c_{15},
b_{12}c_{16})\] in $A=k[a_1, \ldots, a_{11}, b_{12}, c_{13}, \ldots, c_{16}]$. 
The universal family 
is given by the ideal in $(A/J)[x,y,z,w]$ generated by the four elements
\begin{gather*} 
\tilde z\tilde x +b_{12}\tilde a_3\tilde x-b_{12}\tilde a_6\tilde y,\quad 
\tilde z\tilde y  -b_{12}\tilde x(\tilde x+a_8w),\\
\tilde z^2+c_{16}\tilde zw+b_{12}\tilde a_3\tilde z-b_{12}^2\tilde a_6(\tilde x+a_8w)\\
\tilde x^3 +\tilde a_3\tilde x\tilde y-\tilde a_6\tilde y^2+a_8\tilde x^2w
+c_{13}\tilde xw^2+c_{14}\tilde yw^2+c_{15}\tilde zw^2+c_{15}c_{16}w^3,
\end{gather*}
where
 $\tilde x=x+a_2w$, $\tilde y=y+a_1w$, 
$\tilde z=z+a_{11}w+a_9 x+a_{10} y$,
$\tilde a_3=a_3w+a_4 y+a_5x $ and $\tilde a_6=a_6w+a_7 y$.
\end{lemma}

\begin{proof} 
We compute again in the chart  $\{w=1\}$.
By Proposition \ref{tangent of Hilb} there are 16 first order deformations,
which we have to lift to higher order,
as described in \cite[Section I.6]{ArtinTata} and \cite{Stevens}.  We present these pertubations in the following way:
\begin{gather*}
zx +a_2 z+a_9x^2+a_{10}xy+a_{11}x\\
zy  +a_1z+a_9xy+a_{10}y^2+a_{11}y-b_{12}x^2\\
z^2+(c_{16}+2a_{11})z\\
x^3 +(a_3+a_4xy+a_5x)xy-(a_6+a_7y)y^2+(a_8+3a_2)x^2+c_{13}x+c_{14}y+c_{15}z
\end{gather*}
To avoid long formulas we continue the computation with less
variables. First of all, we use coordinate transformations to simplify.
Explicitly, we set $\tilde x=x+a_2$, $\tilde y=y+a_1$, and 
$\tilde z=z+a_{11}+a_9 x+a_{10}y$. This is no loss of generality as one can 
substitute in the final
formulas. Moreover, to ensure a finite computation, we use only
deformations of negative weight. The computation
in \cite{PieneSchlessinger} shows that not all deformation variables
occur in the equations for the base space. 
We put $\tilde a_3=a_3+a_4 y+a_5 x$, $\tilde a_6=a_6+a_7 y$
and  compute with the variables $\tilde a_3$ and $\tilde a_6$,
We give the variable $\tilde z$ weight 2 enabling us to give  $b_{12}$ weight 1. The infinitesimal deformation now becomes 
\begin{gather*}
\tilde z\tilde x, \quad
\tilde z\tilde y  -b_{12}\tilde x^2, \quad
\tilde z^2+c_{16}\tilde z,\\
\tilde x^3 +\tilde a_3\tilde x\tilde y-\tilde a_6\tilde y^2+a_8\tilde x^2
+c_{13}\tilde x+c_{14}\tilde y+c_{15}\tilde z
\end{gather*}

We have to extend these generators of the ideal to higher order (in the 
deformation variables). To ensure flatness we have to lift the equations $\varphi R=0$
to $\widetilde \varphi \widetilde R=0$. We therefore
compute the new relation matrix and then adjust the
generators. The result of our computation is a vector $\widetilde \varphi$ with components
\begin{gather*}
\tilde z\tilde x +b_{12}\tilde a_3\tilde x-b_{12}\tilde a_6\tilde y \\
\tilde z\tilde y  -b_{12}\tilde x(\tilde x+a_8)\\
\tilde z^2+c_{16}\tilde z+b_{12}\tilde a_3\tilde z
         -b_{12}^2\tilde a_6(\tilde x+a_8)\\
\tilde x^3 +\tilde a_3\tilde x\tilde y-\tilde a_6\tilde y^2+a_8\tilde x^2
+c_{13}\tilde x+c_{14}\tilde y+c_{15}\tilde z+c_{15}c_{16}
\end{gather*}
with relation matrix $\widetilde R$
\[
 \begin{pmatrix}
  -b_{12}(\tilde x+a_8)&\tilde z+c_{16}& -\tilde y& -\tilde x(\tilde x+a_8)-c_{13} \\
 -\tilde z-c_{16}-b_{12}\tilde a_3& b_{12}\tilde a_6 & \tilde x & 
        -c_{14}-\tilde a_3\tilde x+\tilde a_6\tilde y \\
 \tilde y & -\tilde x & 0 & -c_{15} \\
 0 & 0 & b_{12} & \tilde z
 \end{pmatrix}\;.
 \]
Modulo terms of degree three in the deformation variables
we have
\[
\widetilde \varphi \widetilde R\equiv
(b_{12}c_{16}\tilde x^2,0,
b_{12}c_{13}\tilde x+b_{12}c_{14}\tilde y+b_{12}c_{15}\tilde z,
b_{12}c_{14}\tilde x^2)\;.
\]

These terms cannot be cancelled by additional higher order terms of $\varphi$,
so the (quadratic) obstruction to lift the deformation is given by the ideal

$J=(b_{12}c_{13},b_{12}c_{14},b_{12}c_{15},
 b_{12}c_{16})$. As the product $\widetilde \varphi \widetilde R$
is in fact equal to zero modulo the ideal $J$, 
the family given is indeed 
the universal family.
\end{proof}

\begin{rem}Compared with \cite{PieneSchlessinger} our formulas differ
 (apart from the names and signs of the variables) in two respects.
 Our $a_5$
is not present in \cite{PieneSchlessinger}, as there it was transformed away
by a coordinate transformation, which does not work in characteristic 3.
But the main difference is the absence of a factor 2 in front of $\tilde a_3$ 
--- it is put there in  \cite{PieneSchlessinger} to complete the square
in the third generator, transforming the term $b_{12}\tilde a_3z$ away; this
cannot be done in characteristic two. 
\end{rem}

\begin{rem} The naming of the deformation variables is chosen such that the variables $a_1, \ldots, a_{11}$ represent deformations in the intersection
of the two components of the Hilbert scheme, the variable $b_{12}$ is in the direction of the twisted
cubic component and the variables $c_{13}, \ldots, c_{16}$ represent deformations in the direction of the other component.
\end{rem}

\section{The moduli space of  stable sheaves}\label{app:FT}
In this section we relate  the moduli space of stable sheaves on $\PP^3$
with Hilbert polynomial $3t+1$ 
to the Hilbert scheme.  The definition of (semi-)stability is recalled in 
\cite[Section 2]{FreiermuthTrautmann}. Note that they work over a 
fixed,  algebraically closed field of characteristic zero.

 Every stable sheaf $\calF$ 
 on $\PP^3$ with Hilbert polynomial $3t+1$ has
a free resolution of the form  \cite{FreiermuthTrautmann}
\[
0 \leftarrow \calF \leftarrow
\calO_{\PP^3}\oplus \calO_{\PP^3}(-1)  \mapleft A
\calO_{\PP^3}(-1) \oplus 3\calO_{\PP^3}(-2) \mapleft B
2 \calO_{\PP^3}(-3)  \leftarrow 0
\]
Any  flat deformation of $\calF$ is obtained by perturbing the
matrices $A$ and $B$  to $\widetilde A$ and $\widetilde B$
such that $\widetilde A\widetilde B=0$.

We consider again the most singular point on the intersection 
of both components of $M^{3t+1}_{\PP^3}$, which is the sheaf
$\calF=i_*\calO_X$ with $(X,i)$ the CM-curve of Lemma \ref{degenerate}.
It has a presentation with matrices
\[
A=\begin{pmatrix}z&0&0&-x^2\\0&z&x&y\end{pmatrix}
\qquad \text{and} \qquad
B=\begin{pmatrix}0&x^2\\-x&-y\\z&0\\0&z\end{pmatrix}.
\]

\begin{lemma}\label{FT section 7} Let $\calF$ be the stable sheaf
with presentation matrix $A$ above.
An affine open chart around the corresponding point in the moduli space 
$M^{3t+1}_{\PP^3}$, 
is given by the ideal $J=(b_{12}c_{13},b_{12}c_{14})$ in 
$k[a_1, \ldots, a_{11}, b_{12}, c_{13},  c_{14}]$. 
The universal family over that affine chart has a presentation
with matrices
\[
\widetilde A= 
\begin{pmatrix}\tilde z &-b_{12}\tilde a_6(\tilde x+a_8) &
       \tilde a_3\tilde x-\tilde a_6\tilde y + c_{14} &-\tilde x(\tilde x+a_8)-c_{13} \\
   -b_{12} & \tilde z+b_{12}\tilde a_3 & \tilde x & \tilde y
\end{pmatrix}\]
and
\[
\widetilde B=
\begin{pmatrix}-\tilde a_3\tilde x+\tilde a_6\tilde y-c_{14} 
      &\tilde x(\tilde x+a_8)+c_{13} \\ -\tilde x & -\tilde y\\
   \tilde z &b_{12}(\tilde x+a_8) \\ b_{12}\tilde a_6 & \tilde z+b_{12}\tilde a_3
\end{pmatrix}
\]
where again
 $\tilde x=x+a_2w$, $\tilde y=y+a_1w$, 
$\tilde z=z+a_{11}w+a_9 x+a_{10} y$,
$\tilde a_3=a_3w+a_4 y+a_5 $ and $\tilde a_6=a_6w+a_7 y$.

The map $M^{3t+1}_{\PP^3}\ra \Hilb_{\PP^3}$ is  given by the identity map on the variables 
$a_1, \ldots, a_{11}, b_{12}, c_{13},  c_{14}$.
\end{lemma}

\begin{proof}
Infinitesimal deformations of the matrices $A$ and $B$ can be found
by direct computation, see also  \cite[Section 7]{FreiermuthTrautmann}.
Then one has to try to lift the equation $AB=0$.
The result is as stated. 
One has
\[
\widetilde A \widetilde B =
\begin{pmatrix}
-c_{13}b_{12}\tilde a_6  & c_{14}b_{12}(\tilde x+a_8)-c_{13}b_{12}\tilde a_3 \\
 c_{14}b_{12} & -c_{13}b_{12}
\end{pmatrix}
\]
Therefore the obstruction to lift the equation is given by the ideal $J=(b_{12}c_{13},b_{12}c_{14})$.

On the component $c_{13}=c_{14}=0$ the map to the Hilbert scheme
is induced by taking the $0^\mathrm{th}$ Fitting ideal. This does not work
on the component $b_{12}=0$. Freiermuth and Trautmann show that
the larger component is isomorphic to the relative universal cubic in the bundle
of hyperplanes of $\PP^3$ \cite[\S6.5]{FreiermuthTrautmann}. The map to the
Hilbert scheme associates to a pair  consisting of a cubic in a hyperplane 
and  a point on it the scheme consisting of this cubic with an embedded point
at this point.
\end{proof}

\begin{rem}
The matrix $\widetilde A$, restricted to the component $c_{13}=c_{14}=0$,
gives a presentation of  $i_{R*}\calO_{X_R}$ with $(X_R,i_R)$ the universal
family of  Lemma \ref{degenerate}.
To see this we use a coordinate transformation from the coordinates in the
Lemma.
We set $\tilde x=x+a_2w$, $\tilde y=y+a_1w$, 
$\tilde z=z-b_{11}w-b_{10} y-b_9 x$,
$\tilde a_3=a_3w+a_4 y+a_5x $ and $\tilde a_6=a_6w+a_7 y$.
We furthermore put $\tilde a_8=a_8-a_2$.
Then the curve becomes in the chart $\{w=1\}$
\[
\begin{pmatrix}
 \tilde x  & \tilde a_6&  u  \\
 \tilde y & u + \tilde a_3 & \tilde x +\tilde a_8
\end{pmatrix}
\]
and the morphism $i_R^*$ is  $(\tilde x, \tilde y,\tilde z)\mapsto
(\tilde x,\tilde y,b_{12}u)$.
\end{rem}


\ifx\undefined\bysame
\newcommand{\bysame}{\leavevmode\hbox to3em{\hrulefill}\,}
\fi

\end{document}